\documentclass[leqno,twocolumn]{revtex4}

\usepackage{xcite}
\usepackage{multirow}

\usepackage{amsthm}
\newcommand{\old}[1]{}

\renewcommand{\emph}[1]{\textit{#1}}
\usepackage{enumerate,amsmath,latexsym,amssymb}

\usepackage{chapterbib} 

\usepackage{pdftricks}
\begin{psinputs}
  \usepackage{pstricks}
\end{psinputs}
\usepackage{color}\usepackage{graphicx}

\definecolor{brown}{cmyk}{0, 0.72, 1, 0.45}
\definecolor{grey}{gray}{0.5}

 \allowdisplaybreaks

\newcommand{\omvar}{\om_{\textrm{var}}}
\newcommand{\ommm}{\om_{\textrm{MM}}}

\newcounter{rot}

\newcommand{\tv}[1]{||#1||_{\rm TV}}

\newcommand{\pmin}{\pi_{\mathrm{min}}}

\newcommand{\ignore}[1]{}

\def\ii_(#1,#2){i_{#1}^{#2}}

\def\cM{\mathcal{M}}
\def\a{\alpha}

\def\d{\delta}

\def\e{\varepsilon}

\def\g{\gamma}

\def\p{\pi}

\def\r{\rho}

\def\s{\sigma}
\def\t{\tau}
\def\om{\omega}

\newcommand{\sbs}{\subseteq}

\def\Re{\mathbb{R}}

\newcommand{\brac}[1]{\left( #1 \right)}

\newcommand{\expect}{\operatorname{\bf E}}
\def\E{\expect}
\renewcommand{\Pr}{\operatorname{\bf Pr}}

\newtheorem{theorem}{Theorem}[section]

\newtheorem{corollary}[theorem]{Corollary}

\newtheorem{observation}[theorem]{Observation}
\newcounter{thmtemp}

\newcommand{\nospace}[1]{}

\def\path{\operatorname{PATH}}

\renewcommand{\Re}{\mathbb{R}}

\newcommand{\flr}[1]{\lfloor #1 \rfloor}

\makeatletter
\newcommand{\arabicifpos}[1]{\@arabicifpos{\@nameuse{c@#1}}}
\newcommand{\@arabicifpos}[1]{\ifnum #1>0 \number #1\fi}
\makeatother

\begin{document}

\title{Assessing significance in a Markov chain without mixing}

\author{Maria Chikina}
\address{Department of Computational and Systems Biology\\
University of Pittsburgh\\
3078 Biomedical Science Tower 3\\
Pittsburgh, PA 15213\\
U.S.A.}
\email[email: ]{mchikina@pitt.edu}
\thanks{Research supported in party by NIH grants 1R03MH10900901A1 and U545U54HG00854003.}
\author{Alan Frieze}
\email[email: ]{alan@random.math.cmu.edu}
\thanks{Research supported in part by NSF Grants DMS1362785, CCF1522984 and a grant(333329) from the Simons Foundation.}
\author{Wesley Pegden}
\email[email: ]{wes@math.cmu.edu}
\thanks{Research supported in part by NSF grant DMS-1363136 and the Sloan foundation.}
\address{Department of Mathematical Sciences\\
Carnegie Mellon University\\
Pittsburgh, PA 15213\\
U.S.A.}

\date{\today}

\begin{abstract}
We present a new statistical test to detect that a presented state of a reversible Markov chain was not chosen from a stationary distribution.  In particular, given a value function for the states of the Markov chain, we would like to demonstrate rigorously that the presented state is an outlier with respect to the values, by establishing a $p$-value under the null hypothesis that it was chosen from a stationary distribution of the chain.

A simple heuristic used in practice is to sample ranks of states from long random trajectories on the  Markov chain, and compare these to the rank of the presented state; if the presented state is a $0.1\%$-outlier compared to the sampled ranks (its rank is in the bottom $0.1\%$ of sampled ranks) then this should correspond to a $p$-value of $0.001$.  This is not rigorous, however, without good bounds on the mixing time of the Markov chain.

Our test is the following: given the presented state in the Markov chain, take a random walk \emph{from the presented state} for any number of steps.  We prove that observing that the presented state is an $\e$-outlier on the walk is significant at $p=\sqrt {2\e}$, under the null hypothesis that the state was chosen from a stationary distribution.  We assume nothing about the Markov chain beyond reversibility, and show that significance at $p\approx\sqrt \e$ is essentially best possible in general.
We illustrate the use of our test with a potential application to the rigorous detection of gerrymandering in Congressional districtings.  


\end{abstract}

\maketitle
\setcitestyle{super}
\footnotetext[0]{\hspace{-1em}Author list is alphabetical.}
\setcitestyle{numbers}
\section{Introduction}
The essential problem in statistics is to bound the probability of a surprising observation, under a \emph{null hypothesis} that observations are being drawn from some unbiased probability distribution.  This calculation can fail to be straightforward for a number of reasons.  On the one hand, defining the way in which the outcome is surprising requires care; for example, intricate techniques have been developed to allow sophisticated analysis of cases where multiple hypotheses are being tested.  On the other hand, the correct choice of the unbiased distribution implied by the null hypothesis is often not immediately clear; classical tools like the $t$-test are often applied by making simplifying assumptions about the distribution in such cases.  If the distribution is well-defined, but not be amenable to mathematical analysis, a $p$-value can still be calculated using bootstrapping, if test samples can be drawn from the distribution.

A third way for $p$-value calculations to be nontrivial occurs when the observation is surprising in a simple way, the null hypothesis distribution is known, but where there is no simple algorithm to draw samples from this distribution.  In these cases, the best candidate method to sample from the null hypothesis is often through a \emph{Markov chain}, which essentially takes a long random walk on the possible values of the distribution.  Under suitable conditions, theorems are available which guarantee that the chain converges to its \emph{stationary distribution}, allowing a random sample to be drawn from a distribution quantifiably close to the target distribution.   This principle has given rise to diverse applications of Markov chains, including to simulations of chemical reactions, to Markov chain Monte Carlo statistical methods, to protein folding, and to statistical physics models.

A persistent problem in applications of Markov chains is the often unknown \emph{rate} at which the chain converges to the stationary distribution\cite{gelman1992inference,Gelman92asingle}.  It is rare to have rigorous results on the mixing time of a real-world Markov chain, which means that in practice, sampling is performed by running a Markov chain for a ``long time'', and hoping that sufficient mixing has occurred.  In some applications, such as in simulations of the Potts model from statistical physics, practitioners have developed modified Markov chains in the hopes of achieving faster convergence \cite{PhysRevLett.58.86}, but such algorithms have still been demonstrated to have exponential mixing times in many settings \cite{borgs2012tight,cooper1999mixing,gore1999swendsen}.

In this paper, we are concerned with the problem of assessing statistical significance in a Markov chain without requiring results on the mixing time of the chain, or, indeed, any special structure at all in the chain beyond reversibility.     Formally, we consider a reversible Markov chain $\cM$ on a state space $\Sigma$, which has an associated label function $\omega:\Sigma\to \Re$.  (The definition of Markov chain is recalled at the end of this section.)  The labels constitute auxiliary information, and are not assumed to have any relationship to the transition probabilities of $\cM$.  We would like to demonstrate that a presented state $\s_0$ is unusual for states drawn from a stationary distribution $\pi$.
If we have good bounds on the mixing time of $\cM$, then we can simply sample from a distribution of $\om(\pi)$, and use bootstrapping to obtain a rigorous $p$-value for the significance of the smallness of the label of $\s_0$.  Such bounds are rarely available, however.   

We propose the following simple and rigorous test to detect that $\s_0$ is unusual relative to states chosen randomly according to $\p$, which does not require bounds on the mixing rate of $\cM$:

\begin{center}
\framebox{\parbox{.96\linewidth}{\textbf{The $\sqrt{\e}$ test:} Observe a trajectory $\s_0,\s_1,\s_2\dots,\s_k$ from the state $\s_0$, for any fixed $k$.  The event that $\om(\s_0)$ is an $\e$-outlier among $\om(\s_0),\dots, \om(\s_k)$ is significant at $p=\sqrt {2\e}$, under the null-hypothesis that $\s_0\sim \p$.}}
\end{center}
Here, we say that a real number $\a_0$ is an \emph{$\e$-outlier} among $\a_0,\a_2,\dots,\a_k$ if there are at most $\e(k+1)$ indices $i$ for which $\a_i\leq \a_0$.  In particular, note for the $\sqrt \e$ test, the only relevant feature of the label function is the ranking it imposes on the elements of $\Sigma$.  In the Supplement, we consider the statistical power of the test, and show that the relationship $p\approx \sqrt{\e}$ is best possible.   We leave as an open question whether the constant $\sqrt 2$ can be improved.

Roughly speaking, this kind of test is possible because a reversible Markov chain cannot have many \emph{local outliers} (Figure \ref{f.simplechain}).  Rigorously, the validity of the test is a consequence of the following theorem.

\begin{theorem}\label{t.gtest}
Let $\cM=X_0,X_1,\dots$ be a reversible Markov chain with a stationary distribution $\p$, and suppose the states of $\cM$ have real-valued labels.  If $X_0\sim \p$, then for any fixed $k$, the probability that the label of $X_0$ is an $\e$-outlier from among the list of labels observed in the trajectory $X_0,X_1,X_2,\dots,X_k$ is at most $\sqrt{2\e}$.
\end{theorem}
\begin{figure}
  \centering
  \includegraphics[width=.8\linewidth]{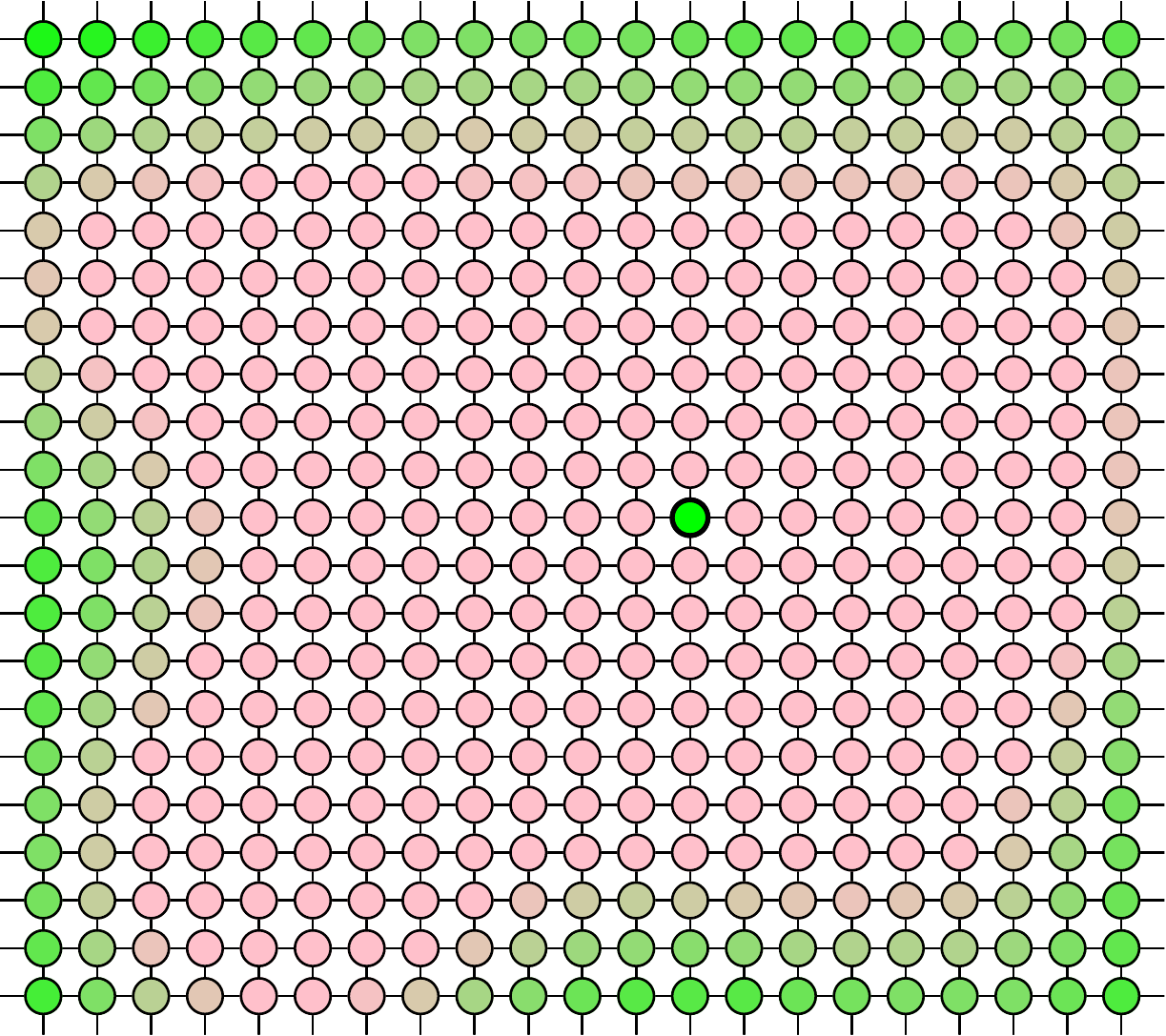}
  \caption{This schematic illustrates a region of a potentially much larger Markov chain with a very simple structure; from each state seen here, a jump is made with equal probabilities to each of the 4 neighboring states.  Colors from green to pink represent labels from small to large.  It is impossible to know from this local region alone whether the highlighted green state has unusually small label in this chain overall.  But to an unusual degree, this state is a \emph{local outlier}.  The $\sqrt \e$ test is based on the fact that \emph{no} reversible Markov chain can have too many local outliers.\label{f.simplechain}}
\end{figure}

We emphasize that Theorem \ref{t.gtest} makes no assumptions on the structure of the Markov chain beyond reversibility.  In particular, it applies even if the chain is not irreducible (in other words, even if the state space is not connected) even though in this case the chain will never mix.

In Section \ref{s.political} we apply the test to Markov chains generating random political districtings, for which no results on rapid mixing exist.  In particular, we show that for various simple choices of constraints on what constitutes a ``valid'' Congressional districting (e.g., that the districts are contiguous, and satisfy certain geometric constraints), the current Congressional districting of Pennsylvania is significantly biased, under the null hypothesis of a districting chosen at random from the set of valid districtings.  (We obtain $p$-values between $\approx 2.5\cdot 10^{-4}$ and $\approx 8.1\cdot 10^{-7}$ for the constraints we considered.)

One hypothetical application of the $\sqrt \e$ test is the possibility of rigorously demonstrating that a chain is not mixed.  In particular, suppose that Research Group 1 has run a reversible Markov chain for $n_1$ steps, and believes that this was sufficient to mix the chain.  Research Group 2 runs the chain for a further $n_2$ steps, producing a trajectory of total length $n_1+n_2$,  and notices that a property of interest changes in these $n_2$ further steps.  Heuristically, this suggests that $n_1$ steps was not sufficient to mix the chain, and the $\sqrt \e$ test quantifies this reasoning rigorously.  For this application, however, we must allow $X_0$ to be distributed not exactly as the stationary distribution $\pi$, but as some distribution $\pi'$ whose total variation distance to $\pi$ is small, as this is the scenario for a ``mixed'' Markov chain.  In the Supplement, we give a version of Theorem \ref{t.gtest} which applies in this scenario.

One area of research related to the present manuscript concerns methods for \emph{perfect sampling} from Markov chains.  Beginning with the Coupling From The Past algorithm of Propp and Wilson\cite{CFTP,CFTPguide} and several extensions\cite{Fill,Huber}, these techniques are designed to allow sampling of states \emph{exactly} from the stationary distribution $\pi$, without having rigorous bounds on the mixing time of the chain.  Compared with $\sqrt \e$ test, perfect sampling techniques have the disadvantage that they require the Markov chain to possess certain structure for the method to be implementable, and that the time it takes to generate each perfect sample is unbounded.  Moreover, although perfect sampling methods do not require rigorous bounds on mixing times to work, they will not run efficiently on a slowly mixing chain.  The point is that for a chain which has the right structure, and which actually mixes quickly (in spite of an absence of a rigorous bound on the mixing time), algorithms like CFTP can be used to rigorously generate perfect samples.    On the other hand, the $\sqrt \e$ test applies to \emph{any} reversible Markov chain, regardless of the structure, and has running time $k$ chosen by the user.  Importantly, it is quite possible that the test can detect bias in a sample even when $k$ is much smaller than the mixing time of the chain, as seems to be the case in the districting example discussed in Section \ref{s.political}.  Of course, unlike perfect sampling methods, the $\sqrt \e$ test can only be used to demonstrate a given sample is not chosen from $\pi$; it does not give a way for generating samples from $\pi$.

\section{Definitions}

We remind the reader that a Markov chain is a discrete time random process; at each step, the chain jumps to a new state, which only depends on the previous state.  Formally, a Markov chain $\cM$ on a state space $\Sigma$ is a sequence $\cM=X_0,X_1,X_2,\dots$ of random variables taking values in $\Sigma$ (which correspond to states which may be occupied at each step) such that for any $\s,\s_0,\dots,\s_{n-1}\in \Sigma$,
\begin{multline*}
\Pr(X_n=\s|X_0=\s_0,X_1=\s_1,\dots,X_{n-1}=\s_{n-1})\\=\Pr(X_1=\s|X_{0}=\s_{n-1}).
\end{multline*}
Note that a Markov chain is completely described by the distribution of $X_0$ and the transition probabilities $\Pr(X_1=\s_1|X_0=\s_0)$ for all pairs $\s_0,\s_1\in \Sigma$.  Terminology is often abused, so that the \emph{Markov chain} refers only to the ensemble of transition probabilities, regardless of the choice of distribution for $X_0$.

With this abuse of terminology, a \emph{stationary distribution} for the Markov chain is a distribution $\pi$ such that $X_0\sim \pi$ implies that $X_1\sim \pi$, and therefore that $X_i\sim \pi$ for all $i$.  When the distribution of $X_0$ is a stationary distribution, the Markov chain $X_0,X_1,\dots$ is said to be \emph{stationary}.  A stationary chain is said to be \emph{reversible} if for all $i,k$, the sequence of random variables $(X_i,X_{i+1},\dots,X_{i+k})$ is identical in distribution to the sequence $(X_{i+k},X_{i+k-1},\dots,X_{i})$.   Finally a chain is \emph{reducible} if there is a pair of states $\sigma_0,\sigma_1$ such that $\sigma_1$ is inaccessible from $\sigma_0$ via legal transitions, and \emph{irreducible} otherwise.

A simple example of a Markov chain is a random walk on a directed graph, beginning from an initial vertex $X_0$ chosen from some distribution.   Here $\Sigma$ is the vertex-set of the directed graph.  If we are allowed to label the directed edges with positive reals and the probability of traveling along an arc is proportional to the label of the arc (among those leaving the present vertex), then any Markov chain has such a representation, as the transition probability $\Pr(X_1=\s_1|X_0=\s_0)$ can be taken as the label of the arc from $\s_0$ to $\s_1$.  Finally, if the graph is undirected, the corresponding Markov chain is reversible.

\section{Detecting bias in political districting}
\label{s.political}
A central feature of American democracy is the selection of Congressional districts in which local elections are held to directly elect national representatives.  Since a separate election is held in each district, the proportions of party affiliations of the slate of representatives elected in a state does not always match the proportions of statewide votes cast for each party.  In practice, large deviations from this seemingly desirable target do occur.

Various tests have been proposed to detect \emph{gerrymandering} of districtings, in which a districting is drawn in such a way as to bias the resulting slate of representatives towards one party; this can be accomplished by concentrating voters of the unfavored party in a few districts.  One class of methods to detect gerrymandering concerns heuristic `smell tests' which judge whether a districting seems generally reasonable in its statistical properties (see, e.g., \cite{Wangthree,nagle}).  For example, such tests may frown upon districtings in which difference between the mean and median vote on district-by-district basis is unusually large \cite{McBest}.

The simplest statistical smell test, of course, is whether the party affiliation of the elected slate of representatives is close in proportion to the party affiliations of votes for representatives.  Many states have failed this simple test spectacularly, such as in Pennsylvania, where in 2012, 48.77\% of votes were cast for Republican representatives and 50.20\% for Democrat representatives, in an election which resulted in a slate of 13 Republican representatives and 5 Democrat representatives.

Heuristic statistical tests such as these all suffer from lack of rigor, however, due to the fact that the statistical properties of `typical' districtings are not rigorously characterized.    For example, it has been shown \cite{unintentional} that Democrats may be at a natural disadvantage when drawing electoral maps even when no bias is at play, because Democrat voters are often highly geographically concentrated in urban areas.  Particularly problematic is that the degree of geographic clustering of partisans is highly variable from state to state: what looks like a gerrymandered districting in one state may be a natural consequence of geography in another.  

Some work has been done in which the properties of a ``valid'' districting are defined (which may be required to have have roughly equal populations among districts, have districts with reasonable boundaries, etc.) so that the characteristics of a given districting can be compared with what would be ``typical'' for a valid districting of the state in question, by using computers to generate random districtings \cite{carolina,minority}; see also \cite{McBest} for discussion.  However, much of this work has relied on heuristic sampling procedures which do not have the property of selecting districtings with equal probability (and, more generally, whose distributions are not well-characterized), undermining rigorous statistical claims about the properties of typical districts.

In an attempt to establish a rigorous framework for this kind of approach, several groups \cite{fifield,cmu,duke} have used Markov chains to sample random valid districtings for the purpose of such comparisons.    
Like many other applications of real-world Markov chains, however, these methods suffer from the completely unknown mixing time of the chains in question.  Indeed, no work has even established that the Markov chains are irreducible (in the case of districtings, this means that any valid districting can be reached from any other by a legal sequence of steps), even if valid districtings were only required to consist of contiguous districts of roughly equal populations.  And, indeed, for very restrictive notions of what constitutes a valid districting, irreducibility certainly fails. 

\smallskip

\begin{figure*}
  \hspace{\stretch{1}}
  \includegraphics[width=.45\linewidth]{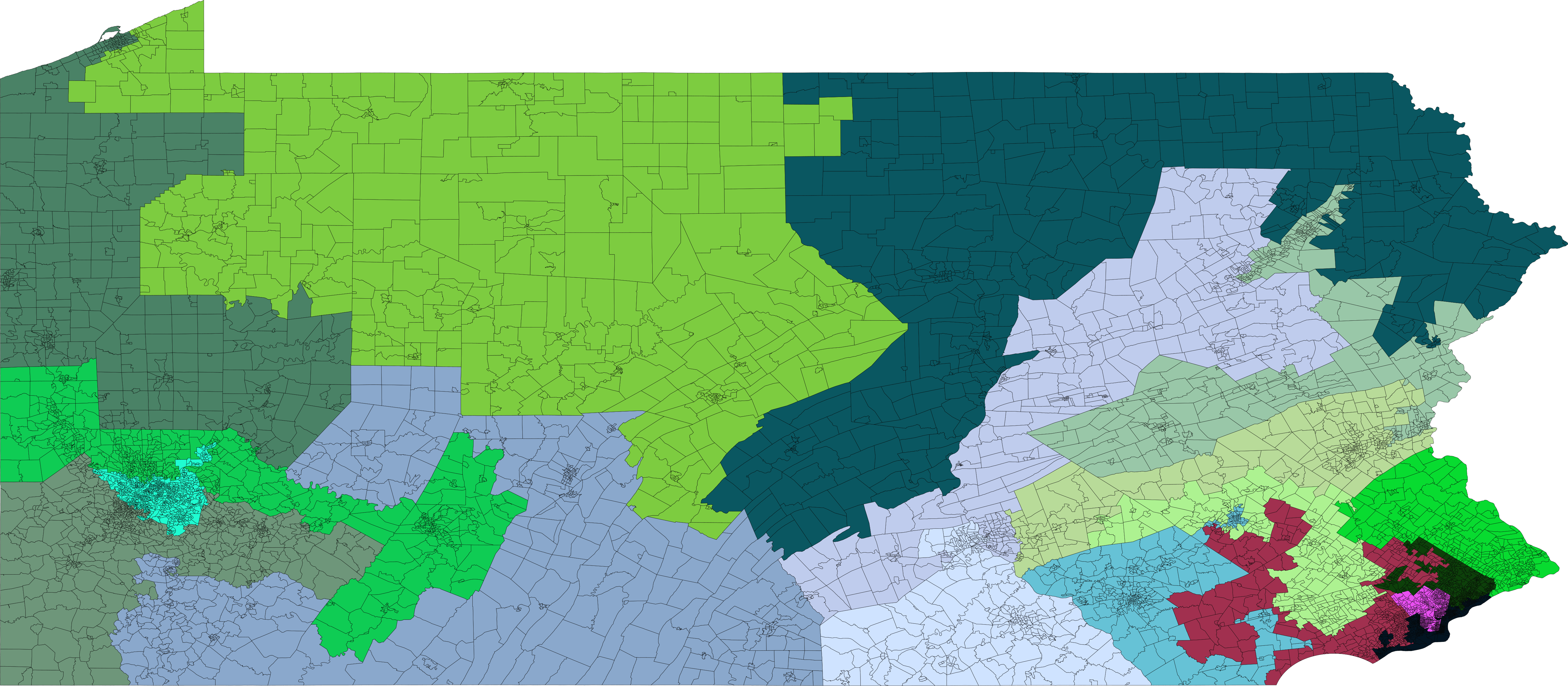}
  \hspace{\stretch{1}}
  \includegraphics[width=.45\linewidth]{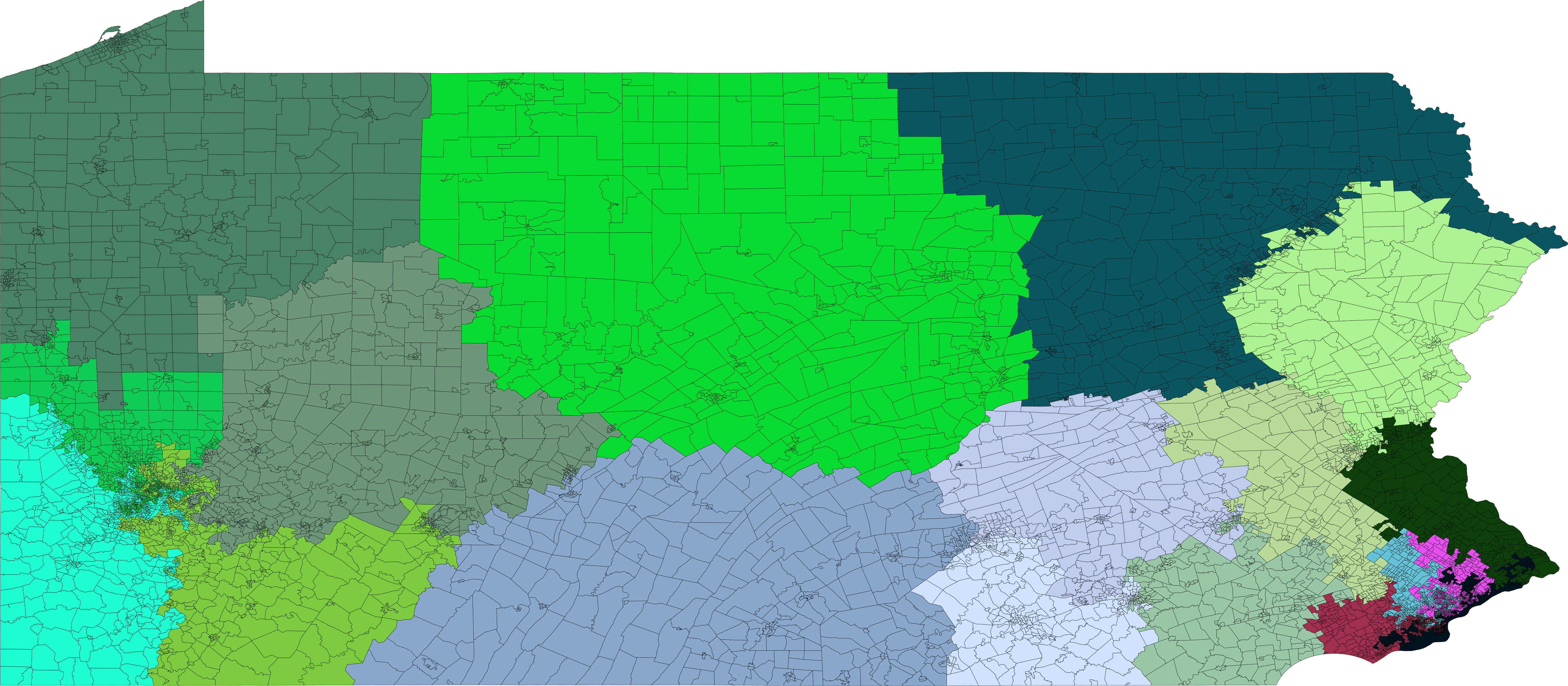}
  \hspace{\stretch{1}}
  \caption{\label{f.globaldistrictings}\textbf{Left:} The current districting of Pennsylvania.  \textbf{Right:} A districting produced by the Markov chain after $2^{40}$ steps.  (Detailed parameters for this run are given in the supplement.)}
\end{figure*}

As a straightforward application of the $\sqrt \e$ test, we can achieve rigorous $p$-values in Markov models of political districtings in spite of the lack of bounds on mixing times of the chains.  In particular, for all choices of the constraints on valid districtings we tested, the $\sqrt \e$ test showed that the current Congressional districting of Pennsylvania is an outlier at significance thresholds ranging from $p\approx 2.5\cdot 10^{-4}$ and $p\approx 8.1\cdot 10^{-7}$.   Detailed results of these runs are in the Supplement.

A key advantage of the Markov chain approach to gerrymandering is that it rests on a rigorous framework; namely, comparing the actual districting of a state with typical (i.e., random) districtings from a well-defined set of valid districtings.  The rigor of the approach thus depends on the availability of a precise definition of what constitutes a valid districting; in principle and in practice, this is a a thorny legal question.  While some work on Markov chains for redistricting (in particular, \cite{duke}) has aimed to account for complex constraints on valid districtings, our main goal in the present manuscript is to illustrate the application of the $\sqrt \e$ test.  In particular, we have erred on the side of using relatively simple sets of constraints on valid districtings in our Markov chains, while checking that our significance results are not highly sensitive to the parameters that we use.  On the other hand, our test immediately gives a way of putting the work such as that in \cite{duke} on a rigorous statistical footing.

The full description of the Markov chain we use in the present work is given in the supplement, but it's basic structure is as follows: Pennsylvania is divided into roughly 9000 Census blocks.   (These blocks can be seen upon close inspection of Figure \ref{f.globaldistrictings}.)  We define a division of these blocks into 18 districts to be a valid districting of Pennsylvania if districts differ in population by less than $2\%$, are contiguous, are simply connected (districts do not contain holes) and are ``compact'' in ways we discuss in the supplement; roughly, this final condition prohibits districts with extremely contorted structure.  The state space of the Markov chain is the set of valid districtings of the state, and one step of the Markov chain consists of randomly swapping a precinct on the boundary of a district to a neighboring district, if the result is still a valid districting.  As we discuss in the supplement, the chain is adjusted slightly to ensure that the uniform distribution on valid districtings is indeed a stationary distribution for the chain.  Observe that this Markov chain has a potentially huge state space; if the only constraint on valid districtings was that the districts have roughly equal population, there would be $10^{10000}$ or so valid districtings.  Although contiguity and especially compactness are severe restrictions which will decrease this number substantially, it seems difficult to compute effective upper bounds on the number of resulting valid districtings, and certainly, it is still enormous.  Impressively, these considerations are all immaterial to our very general method.

Applying the $\sqrt \e$ test involves the choice of a label function $\om(\sigma)$, which assigns a real-number to each districting.  We have conducted runs using two label functions: $\omvar$ is the (negative) variance of the proportion of Democrats in each district of the districting (as measured by 2012 presidential votes), and $\ommm$ is the difference between the median and mean of the proportions of Democrats in each district.  $\ommm$ is motivated by the fact that this metric has a long history of use in gerrymandering, and is directly tied to the goals of gerrymandering, while the use of the variance is motivated by the fact that it can change quickly with small changes in districtings.  These two choices are discussed further in the Supplement, but an important point is that our use of these label functions \textbf{is not} based on an assumption that small values of $\omvar$ or of $\ommm$ directly imply gerrymandering.  Instead, as Theorem \ref{t.gtest} is valid for any fixed label function, these labels are tools used to demonstrate significance, which are chosen because they are simple and natural functions on vectors which can be quickly computed, seem likely to be different for typical versus gerrymandered districtings, and have the potential to change relatively quickly with small changes in districtings.   For the various notions of valid districtings we considered, the $\sqrt{\e}$ test demonstrated significance at $p$-values in the range $10^{-4}$ to $10^{-5}$ for the $\ommm$ label function, and in the range $10^{-4}$ and $10^{-7}$ for the $\omvar$ label function.   

As noted earlier, the $\sqrt \e$ test can easily be used with more complicated Markov chains which capture more intricate definitions of the set of valid districtings.  For example, the current districting of Pennsylvania splits fewer rural counties than the districting on the right in Figure \ref{f.globaldistrictings}, and the number of county splits is one of many metrics for valid districtings considered by the Markov chains developed in \cite{duke}.  Indeed, our test will be of particular value in cases where complex notions of what constitute a valid districtings slow the chain to make the heuristic mixing assumption particularly questionable.  Regarding mixing time: even our chain with relatively weak constraints on the districtings (and very fast running time in implementation) appears to mix too slowly to sample $\pi$, even heuristically; in Figure \ref{f.globaldistrictings}, we see that several districts seem still to have not left their general position from the initial districting, even after $2^{40}$ steps.

On the same note, it should also be kept in mind that while our result gives a method to rigorously disprove that a given districting is unbiased---e.g., to show that the districting is unusual among districtings $X_0$ distributed according to the stationary distribution $\pi$---it does so \emph{without} giving a method to sample from the stationary distribution.  In particular, our method can not answer the question of how many seats Republicans and Democrats should have in a typical districting of Pennsylvania, because we are still not mixing the chain.  Instead, Theorem \ref{t.gtest} has given us a way to disprove $X_0\sim \pi$ without sampling $\pi$.

\section{Proof of Theorem \ref{t.gtest}}
We let $\pi$ denote any stationary distribution for $\cM$, and suppose that the initial state $X_0$ is distributed as $X_0\sim \p$, so that in fact $X_i\sim \pi$ for all $i$.   We say $\s_j$ is \emph{$\ell$-small} among $\s_0,\dots,\s_k$ if there are at most $\ell$ indices $i\neq j$ among $0,\dots,k$ such that the label of $\s_i$ is at most the label of $\s_j$.  In particular, $\s_j$ is $0$-small among $\s_0,\s_1,\dots,\s_k$ when its label is the unique minimum label, and we encourage readers to focus on this $\ell=0$ case in their first reading of the proof.

For $0\leq j\leq k$, we define
\begin{align*}
\r_{j,\ell}^k&:=\Pr\left(X_j\mbox{ is $\ell$-small among }X_0,\dots,X_k\right)\\
\r_{j,\ell}^k(\s)&:=\Pr\left(X_j\mbox{ is $\ell$-small among }X_0,\dots,X_k\mid X_j=\s\right)
\end{align*}

Observe that since $X_s\sim \pi$ for all $s$, we also have that
\begin{multline}\label{l.lshift}
\r_{j,\ell}^k(\s)=\\\Pr\left(X_{s+j}\mbox{ is $\ell$-small among }X_s,\dots,X_{s+k}\mid X_{s+j}=\s\right)
\end{multline}

We begin by noting two easy facts.
\begin{observation}\label{o.rev}
$\r_{j,\ell}^k(\s)=\r_{k-j,\ell}^k(\sigma)$.
\end{observation}

\begin{proof}
Since $\cM=X_0,X_1,\dots$ is stationary and reversible, the probability that $(X_0,\dots,X_k)=(\s_0,\dots,\s_k)$ is equal to the probability that $(X_0,\dots,X_k)=(\s_k,\dots,\s_0)$ for any fixed sequence $(\s_0,\dots,\s_k)$.  Thus, any sequence $(\s_0,\dots,\s_k)$ for which $\s_j=\s$ and $\s_j$ is a $\ell$-small corresponds to an equiprobable sequence $(\s_k,\dots,\s_0)$ for which $\s_{k-j}=\s$ and $\s_{k-j}$ is $\ell$-small.
\end{proof}

\begin{observation}\label{o.split}
$\r_{j,2\ell}^k(\s)\geq \r_{j,\ell}^j(\s)\cdot \r_{0,\ell}^{k-j}(\s).$
\end{observation}
\begin{proof}
Consider the events that $X_j$ is a $\ell$-small among $X_0,\dots,X_j$ and among $X_j,\dots,X_k$, respectively.  These events are conditionally independent, when conditioning on the value of $X_j=\s$, and $\r_{j,\ell}^j(\s)$ gives the probability of the first of these events, while applying equation \eqref{l.lshift} with $s=j$ gives that $\r_{0,\ell}^{k-j}(\s)$ gives the probability of the second event.

Finally, when both of these events happen, we have that $X_j$ is $2\ell$-small among $X_0,\dots,X_k$.
\end{proof}

We can now deduce that
\begin{equation}\label{l.rs}
\r_{j,2\ell}^k(\s)\geq\rho_{j,\ell}^j(\s)\cdot \rho_{0,\ell}^{k-j}(\s)\\=
\r_{0,\ell}^j(\s)\cdot \r_{0,\ell}^{k-j}(\s)\geq \left(\r_{0,\ell}^{k}(\s)\right)^2.
\end{equation}
Indeed, the first inequality follows from Observation \ref{o.split}, the equality follows from Observation \ref{o.rev}, and the final inequality follows from the fact that $\r_{j,\ell}^k(\s)$ is monotone nonincreasing in $k$ for fixed $j,\ell,\s$.

 Observe now that 
$
\r_{j,\ell}^k=\E \r_{j,\ell}^k(X_j),
$
where the expectation is taken over the random choice of $X_j\sim \pi$.

Thus taking expectations in \eqref{l.rs} we find that
\begin{multline}\label{l.exp}
\r_{j,2\ell}^k=\E\r_{j,2\ell}^k(\s)\geq \E\left(\left(\r_{0,\ell}^{k}(\s)\right)^2\right)\\\geq \left(\E\r_{0,\ell}^{k}(\s)\right)^2=(\r_{0,\ell}^{k})^2.
\end{multline}
where the second of the two inequalities is the Cauchy-Schwartz inequality.

For the final step in the proof, we sum the left and righthand sides \eqref{l.exp} to obtain 
\[
\sum_{j=0}^k \rho_{j,2\ell}^k\geq(k+1) (\rho_{0,\ell}^{k})^2
\]
If we let $\xi_j$ $(0\leq i\leq k)$ be the indicator variable which is 1 whenever $X_j$ is $2\ell$-small among $X_0,\dots,X_k$, then $\sum_{j=0}^k{\xi_j}$ is the number of $2\ell$-small terms, which is always at most $2\ell+1$, so that linearity of expectation gives that
\begin{equation}\label{l.sums}
2\ell+1\geq (k+1) (\r_{0,\ell}^k)^2,
\end{equation}
giving that 
\begin{equation}\label{l.final}
\r_{0,\ell}^k\leq\sqrt{\tfrac{2\ell+1}{k+1}}.
\end{equation}
This proves Theorem \ref{t.gtest}, as if $X_i$ is an $\e$-outlier among $X_0,\dots,X_{k}$, then $X_i$ is necessarily $\ell$-small among $X_0,\dots,X_{k}$ for $\ell=\flr{\e(k+1)-1}\leq \e(k+1)-1$, and then we have $2\ell+1\leq 2\e(k+1)-1\leq 2\e(k+1)$.\qed

\bigskip

\subsection*{Acknowledgment}
We are grateful for helpful conversations with John Nagle, Danny Sleator, and Dan Zuckerman.

\bibliographystyle{plain}
\bibliography{outliers}

\widetext
\clearpage

\begin{center}
\textbf{\large Supplement: Assessing significance without mixing}
\end{center}
\setcounter{equation}{0}
\setcounter{figure}{0}
\setcounter{table}{0}
\setcounter{page}{1}
\setcounter{section}{0}
\makeatletter
\renewcommand{\theequation}{S\arabic{equation}}
\renewcommand{\thefigure}{S\arabic{figure}}
\renewcommand{\thesection}{S\arabic{section}}

\section{Precinct data}
Precinct level voting data and associated shape files were obtained from the Harvard Election Data Archive (\url{http://projects.iq.harvard.edu/eda/home}) \cite{harvard}.  The data for Pennsylvania contains 9256 precincts.  The data was altered in two ways: 258 precincts that were contained within another precinct were merged and 79 precincts that were not contiguous were split into continuous areas, with voting and population data distributed proportional to the area.  The result was a set of 9060 precincts.  All geometry calculations and manipulations were accomplished in R with ``maptools'', ``rgeos'', and ``BARD'' R packages.  The final input to the Markov chain is a set of precincts with corresponding areas, neighbor precincts, the length of the perimeter shared with each neighbor, voting data from 2012, and the current Congressional district the precinct belongs to. 

\section{Valid districtings}
We restrict our attention to districtings satisfying 4 restrictions, each of which we describe here.

\subsection{Contiguous districts}
A valid districting must have the property that each of its districts is contiguous.  In particular, two precincts are considered adjacent if the length of their shared perimeter is nonzero (in particular, precincts meeting only at a point are \emph{not} adjacent), and a district is contiguous if any pair of precincts is joined by a sequence of consecutively adjacent pairs.

\subsection{Simply connected districts}
A valid districting must have the property that each of its districts is simply connected.  Roughly speaking, this means the district cannot have a ``hole''.  Precisely, a district is simply connected if for any circuit of precincts in the district, all precincts in the region bounded by the circuit also lie in the district.

Apart from aesthetic reasons for insisting that districtings satisfy this condition, there is also a practical reason: it is easier to have a fast local check for contiguity when one is also maintaining that districtings are simply connected.

\subsection{Small population difference}
According to the ``one person, one vote'' doctrine, Congressional districts for a state are required to be roughly equal in population.  In the current districting of Pennsylvania, for example, the maximum difference in district population from the average population is roughly $1\%$.  Our chain can use different tolerances for population difference between districts and the average, and the tolerances used in the runs below are indicated.

\subsection{Compactness}
If districtings were drawn randomly with only the first three requirements, the result would be districtings in which districts have very complicated, fractal-like structure (since most districtings have this property).  The final requirement on valid districtings prevents this, by ensuring that the districts in the districting have a reasonably nice shape.  This requirement on district shape is commonly termed ``compactness'', is explicitly required of Congressional districts by the Pennsylvania constitution.

Although compactness of districts does not have a precise legal definition, various precise metrics have been proposed to quantify the compactness of a given district mathematically.  One of the simplest and most commonly used metrics is the Polsby-Popper metric, which defines the compactness of a district as
\[
C_D=\frac{4\pi A_D}{P_D^2},
\]
where $A_D$ and $P_D$ are the area and perimeter of the district $D$.  Note that the maximum value of this measure is 1, which is achieved only by the disc, as a result of the isoperimetric inequality.  All other shapes have compactness between 0 and 1, and smaller values indicate more ``contorted'' shapes.

Perhaps the most straightforward use of this compactness measure is to enforce some threshold on compactness, and require valid districtings to have districts whose compactness is above that lower bound.  (For consistency with our other metrics, we actually impose an upper bound on the reciprocal $1/C_D$ of the Polsby-Popper compactness $C_D$ of each district $D$.) In the table of runs given in Section \ref{s.runs}, this is the $L^\infty$ compactness metric.

One drawback of using this method is that the current districting of Pennsylvania has a few districts which have very low compactness values (they are much stranger looking than the other districts). Applying this restriction will allow all 18 districts to be as bad as the threshold chosen, so that, in particular, we will be sampling districtings from space in which all 18 districts may be as bad as the worst district in the current plan.  In fact, because there are more noncompact regions than compact ones, one expects that in a typical such districting, all 18 districts would be nearly as bad as the currently worst example. 

To address this issue, and also to demonstrate the robustness of our finding for the districting question, we also consider some alternate restrictions on the districting, which measure how reasonable the districting as a whole is with regard to compactness.  For example, one simple measure of this is to have a threshold for the maximum allowable sum
\[
\frac 1 {C_1}+\dots+\frac 1 {C_{18}}
\]
of the inverse compactness values of the 18 districts.  This is the $L^1$ metric in the table in Section \ref{s.runs}.  Similarly, we could have a threshold for the maximum allowable sum of squares
\[
\frac 1 {C_1^2}+\dots+\frac 1 {C_{18}^2}.
\]
This is the $L^2$ metric in the table.  Finally, we can have a simple condition which simply ensures that the total perimeter 
\[
P_1+\dots+P_{18}
\]
is less than some threshold.


\subsection{Other possible constraints}
It is possible to imagine many other reasonable constraints on valid districtings.  For example, the PA constitution currently requires of districtings for the Pennsylvania Senate and Pennsylvania House of Representatives that:
\begin{quote}
Unless absolutely necessary no county, city, incorporated town, borough, township or ward shall be divided in forming either a senatorial or representative district.
\end{quote}
There is no similar requirement for U.S. Congressional districts in Pennsylvania, which is what we consider here, but it is still a reasonable constraint to consider.

There are also interesting legal-questions about the extent to which Majority-Minority districts (in which an ethnic minority is an electoral majority) are either required to be intentionally created, or forbidden to be intentionally created. On the one hand, the U.S.  Supreme Court ruled in \emph{Thornburg v.~Gingles} (1986) that in certain cases, a geographically concentrated  minority population is entitled to a Congressional district in which it constitutes a majority.   On the other hand, several U.S. Supreme Court cases (\emph{Shaw v.~Reno} (1993), \emph{Miller v.~Johnson} (1995), and \emph{Bush v.~Vera} (1996)) Congressional districtings were thrown out because because they contained intentionally-drawn Majority-Minority districts which were deemed to be a ``racial gerrymander''.  In any case, we have not attempted to answer the question of whether or how the existence of Majority-Minority districts should be quantified.  (We suspect that the unbiased procedure of drawing a random districting is probably acceptable under current Majority-Minority district requirements, but in any case, our main intent is to demonstrate the application of the $\sqrt \e$ test.)

Importantly, we emphasize that any constraint on districtings which can be precisely defined (i.e., by giving an algorithm which can identify whether a districting satisfies the constraint) can be used in the Markov Chain setting in principle.

\section{The Markov Chain}
The Markov chain $\cM$ we use has as its state space $\Sigma$ the space of all valid districtings (with 18 districts) of Pennsylvania.  Note that there is no simple way to enumerate these, and there is an enormous number of them.

A simple way to define a Markov chain on this state space is to transition as follows:
\begin{enumerate}
\item From the current state, determine the set $S$ of all pairs $(\rho,D)$, where $\rho$ is a precinct in some district $D_\rho$, and $D\neq D_\rho$ is a district which is adjacent to $\rho$.  Let $N_S$ denote the size of this set.
\item From $S$, choose one pair $(\rho_0,D_0)$ uniformly at random.
\item Change the district membership of $\rho_0$ from $D_{\rho_0}$ to $D_0$, \textbf{if} the resulting district is still valid.
\end{enumerate}

Let the Markov chain with these transition rules be denoted by $\cM_0$.  This is a perfectly fine reversible Markov Chain to which our theorem applies, but the uniform distribution on valid districtings is not stationary for $\cM_0$, and so we cannot use $\cM_0$ to make comparisons between a presented districting and a uniformly random valid districting. 

\newcommand{\nmax}{N_{\rm max}}
A simple way to make the uniform distribution stationary is to ``regularize'' the chain, that is, to modify the chain so that the number of legal steps from any state is equal.  (This is not the case for $\cM_0$, as the number of precincts on the boundaries of districts will vary from districting to districting.)  We do this by adding loops to each possible state.  In particular, using a theoretical maximum $\nmax$ on the possible size of $N_S$ for any district, we modify the transition rules as follows:
\begin{enumerate}
\item From the current state, determine the set $S$ of all pairs $(\rho,D)$, where $\rho$ is a precinct in some district $D_\rho$, and $D\neq D_\rho$ is a district which is adjacent to $\rho$.  Let $N_S$ denote the size of this set.
\item \label{s.loop} With probability $1-\frac{N_S}{\nmax}$, remain in the current state for this step.  With probability $\frac{N_S}{\nmax}$, continue as follows:
\item From $S$, choose one pair $(\rho_0,D_0)$ uniformly at random.
\item \label{s.swap} Change the district membership of $\rho_0$ from $D_{\rho_0}$ to $D_0$, \textbf{if} the resulting district is still valid.  If it is not, remain in the current district for this set.
\end{enumerate}
In particular, with this modification, each state has exactly $\nmax$ possible transitions, which are each equally likely; many of these transitions are loops back to the same state.  (Some of these loops arise from Step \ref{s.loop}, but some also arise when the \textbf{if} condition in Step \ref{s.swap} fails.)

\section{The label function}
In principle, any label function $\om$ could be used in the application of the $\sqrt \e$ test; note that Theorem \ref{t.gtest} places no restrictions on $\omega$.  Thus when we choose which label function to use, we are making a choice based on what is likely to achieve good significance, rather than what is valid statistical reasoning (subject to the caveat discussed in the last paragraph of this section).  To choose a label function which was likely to allow good statistical power, we want to have a function which:
\begin{enumerate}
\item is likely very different for a gerrymandered districting compared with a typical districting, and
\item is sensitive enough that small changes in the districting might be detected in the label function.
\end{enumerate}
While the role of the first condition in achieving outlier status is immediately obvious, the second property is also crucial to detecting significance with our test, which makes use of trajectories which may be quite small compared with the mixing time of the chain.  For the $\sqrt{\e}$ test to succeed at demonstrating significance, it is not enough for the presented state $\s_0$ to actually be an outlier against $\pi$ with respect to $\om$; this must also be detectable on trajectories of the fixed length $k$, which may well be too small to mix the chain.  This second property discourages the use of ``coarse grained'' label functions such as the number of seats out of 18 the Democrats would hold with the districting in question, since many swaps would be needed to shift a representative from one party to another.

We considered two label functions for our experiments (each selected with the above desired properties in mind) to demonstrate the robustness of our framework.  The first label function $\omvar$ we used is simply the negative of the variance in the proportions of Democrat voters among the districts.  Thus, given a districting $
\sigma$, $\omvar(\sigma)$ is calculated as
\[
\omvar(\s)=-\brac{\frac{\delta_1^2+\delta_2^2+\dots +\delta_{18}^2}{18}-\left(\frac{\delta_1+\delta_2+\dots +\delta_{18}}{18}\right)^2}
\]
where for each $i=1,\dots,18$, $\delta_i$ is the fraction of voters in that district which voted for the Democrat presidential candidate in 2012.  We suspect that the variance is a good label function from the standpoint of the first characteristic listed above, but a great label function from the standpoint of the second characteristic.  Recall that in practice, gerrymandering is accomplished by packing the voters of one party into a few districts, in which they make up an overwhelming majority.  This, naturally, results in a high-variance vector of party proportions in the districts.  However, high-variance districtings can exist which do not favor either party (note, for example, that the variance is symmetric with respect to Democrats and Republicans, ignoring third-party affiliations).  This means that for a districting which is biased against $\pi$ due to a partisan gerrymander to ``stand out'' as an outlier, it must have especially high variance.  In particular, statistical significance will be weaker than it might be for a label function which is more strongly correlated with partisan gerrymandering.  On the other hand, $\omvar$ can detect very small changes in the districting, since essentially every swap will either increase or decrease the variance.  Indeed, for the run of the chain corresponding to the $L^\infty$ constraint (Section \ref{s.runs}), $\omvar(X_0)$ was strictly greater than $\omvar(X_i)$ for the entire trajectory $(1\leq i\leq 2^{40})$.  That is, for the $L^{\infty}$ constraint, the current districting of Pennsylvania was the absolute worst districting seen according to $\omvar$ among the more than trillion districtings on the trajectory. 

The second label function we considered is calculated simply as the difference between the median and the mean of the ratios $\d_1,\dots,\d_{18}$.  This simple metric, called the ``Symmetry Vote Bias'' by McDonald and Best \cite{McBest} and denoted as $\ommm$ by us, is closely tied to the goal of partisan gerrymandering.  As a simple illustration of the connection, we consider the case where the median of the ratios $\d_1,\dots,\d_{18}$ is close to $\tfrac 1 2$.  In this case, the mean of the $\d_i$'s tracks the fraction of votes the reference party wins in order to win half the seats.  Thus a positive Symmetry Vote Bias corresponds to an advantage for the reference party, while a negative Symmetry Vote Bias corresponds to an disadvantage.  The use of the Symmetry Vote Bias in evaluating districtings dates at least to the 19th century work of Edgeworth \cite{edgeworth}.  These features make it an excellent candidate from the standpoint of our first criterion: gerrymandering is very likely to be reflected in outlier-values of $\ommm$.

On the other hand, $\ommm$ is a rather slow-changing function, compared with $\omvar$.  To see this, observe that in the calculation
\[
\textrm{Symmetry Vote Bias}=\textrm{median}-\text{mean},
\]
the mean is essentially fixed, so that changes in $\ommm$ depend on changes in the median.  In initial changes to the districting, only changes to the two districtings giving rise to the median (two since 18 is even) can have a significant impact on $\ommm$.  (On the other hand, changes to any district directly affect $\omvar$.)

It is likely possible to make better choices for the label function $\om$ to achieve better significance.  For example, the metric $B_G$ described by Nagle \cite{nagle} seems likely to be excellent from the standpoints of conditions 1 and 2 simultaneously.  However, we have restricted ourselves to the simple choices of $\omvar$ and $\ommm$ to clearly demonstrate our method, and to make it obvious that we have not tried many labeling functions before finding some that worked (in which case, a multiple hypothesis test would be required).

\smallskip
One point to keep in mind is that often when applying the $\sqrt \e$ test---including in the present application to gerrymandering---we will wish to apply the test, and thus need to define a label function, after the presented state $\s_0$ is already known.  In these cases, care must be taken to choose a ``canonical'' label function $\om$, so that there is no concern that $\om$ was carefully crafted in response to $\s_0$ (in this case, a multiple hypothesis correction would be required, for the various possible $\om$'s which could have been crafted, depended on $\s_0$).  $\omvar$ and $\ommm$ are excellent choices from this perspective: the variance is a common and natural function on vectors, and the Symmetry Vote Bias has an established history in the evaluation of gerrymandering (and in particular, a history which predates the present districting of Pennsylvania).

\section{Runs of the chain}
\label{s.runs}
In Table \ref{t.runs} we give the results of the 8 runs of the chain under various conditions.  Each run was for $k=2^{40}$ steps.  Code and input data for our Markov chain are available at the corresponding author's website (\url{http://math.cmu.edu/~wes}).

\begin{table}[b]
\renewcommand{\arraystretch}{1.2}
\footnotesize
\begin{center}
  \begin{tabular}{c|c|c|c|@{\hskip .1cm}c@{\hskip .1cm}||@{\hskip .1cm}c@{\hskip .1cm}||c|c}
     \shortstack{population\\ threshold} &  \shortstack{compactness \\ measure} &  \shortstack{compactness \\ threshold} &  \shortstack{initial \\ value} &  \shortstack{(steps)\\$k=$} &  \shortstack{label\\function} &  \shortstack{$\e$-outlier\\ at $\e=$} &  \shortstack{significant\\ at $p=$}\\
    \hline \hline
    \multirow{2}{*}{$2\%$} & \multirow{2}{*}{perimeter} & \multirow{2}{*}{$\leq 125$} &  \multirow{2}{*}{$121.2\dots$} & \multirow{2}{*}{$2^{40}$} & $\omvar$ & $3.0974\cdot 10^{-8}$ &$2.4889\cdot 10^{-4}$\\
    &  &  &  &  & $\ommm$ & $5.7448\cdot 10^{-10}$ &$3.3896\cdot 10^{-5}$\\
\hline
    \multirow{2}{*}{$2\%$} & \multirow{2}{*}{$L^1$} & \multirow{2}{*}{$\leq 160$} &  \multirow{2}{*}{$156.4\dots$} & \multirow{2}{*}{$2^{40}$}  & $\omvar$ & $5.0123\cdot 10^{-11}$ &$1.0012\cdot 10^{-5}$\\
    &  &  &  &  & $\ommm$ & $5.6936\cdot 10^{-10}$ & $3.3745\cdot 10^{-5}$\\
\hline
    \multirow{2}{*}{$2\%$} & \multirow{2}{*}{$L^2$} & \multirow{2}{*}{$\leq 44$} &  \multirow{2}{*}{$43.06\dots$} & \multirow{2}{*}{$2^{40}$} & $\omvar$  & $8.2249\cdot 10^{-11}$ &$1.2826\cdot 10^{-5}$\\
   & & & & & $\ommm$ & $6.8038\cdot 10^{-10}$ &$3.6888\cdot 10^{-5}$\\
\hline
    \multirow{2}{*}{$2\%$} & \multirow{2}{*}{$L^\infty$} & \multirow{2}{*}{$\leq 25$} &  \multirow{2}{*}{$24.73\dots$} & \multirow{2}{*}{$2^{40}$} & $\omvar$ & $3.3188\cdot 10^{-13}$ &$8.1472\cdot 10^{-7}$\\
  & & & & & $\ommm$ & $6.9485\cdot 10^{-8}$ &$3.7279\cdot 10^{-4}$\\
  \end{tabular}
\end{center}
\caption{\label{t.runs} Runs of the redistricting Markov Chain, with results of the $\sqrt\e$ test.}
\end{table}

Generally, after an initial ``burn-in'' period, we expect the chain to (almost) never again see states as unusual as the current districting of Pennsylvania, which means that we expect the test to demonstrate significance proportional to the inverse of the square root of the number of steps (i.e., the $p$-value at $2^{42}$ steps should be half the $p$-value at $2^{40}$ steps).  In particular, for the $L^1$, $L^2$, and $L^\infty$ constraints, these runs never revisited states as bad as the initial state after $2^{21}$ steps for the $\ommm$ label, and after $2^6$ steps for the $\omvar$ label.  Note that this agrees with our guess that $\omvar$ had the potential to change more quickly than $\ommm$. The perimeter constraint did revisit enough states as bad as the the given state with respect to the $\omvar$ label to adversely affect its $p$-value with respect to the $\omvar$ label.  This may reflect our guess that the $\omvar$ label is worse than the $\ommm$ label in terms of how easily it can distinguish gerrymandered districtings from random ones.

The parameters for the first row were used for Fig 2 of the paper.  

One quick point: although we have experimented here with different compactness measures, there is no problem of multiple hypothesis correction to worry about, as \emph{every} run we encountered produces strong significance for the bias of the initial districting.  The point of experimenting with the notion of compactness is to demonstrate that this a robust framework, and that the finding is unlikely to be sensitive to minor disagreements over the proper definition of the set of valid districtings.

\begin{figure}
  \begin{center}
  \includegraphics[width=.24\linewidth]{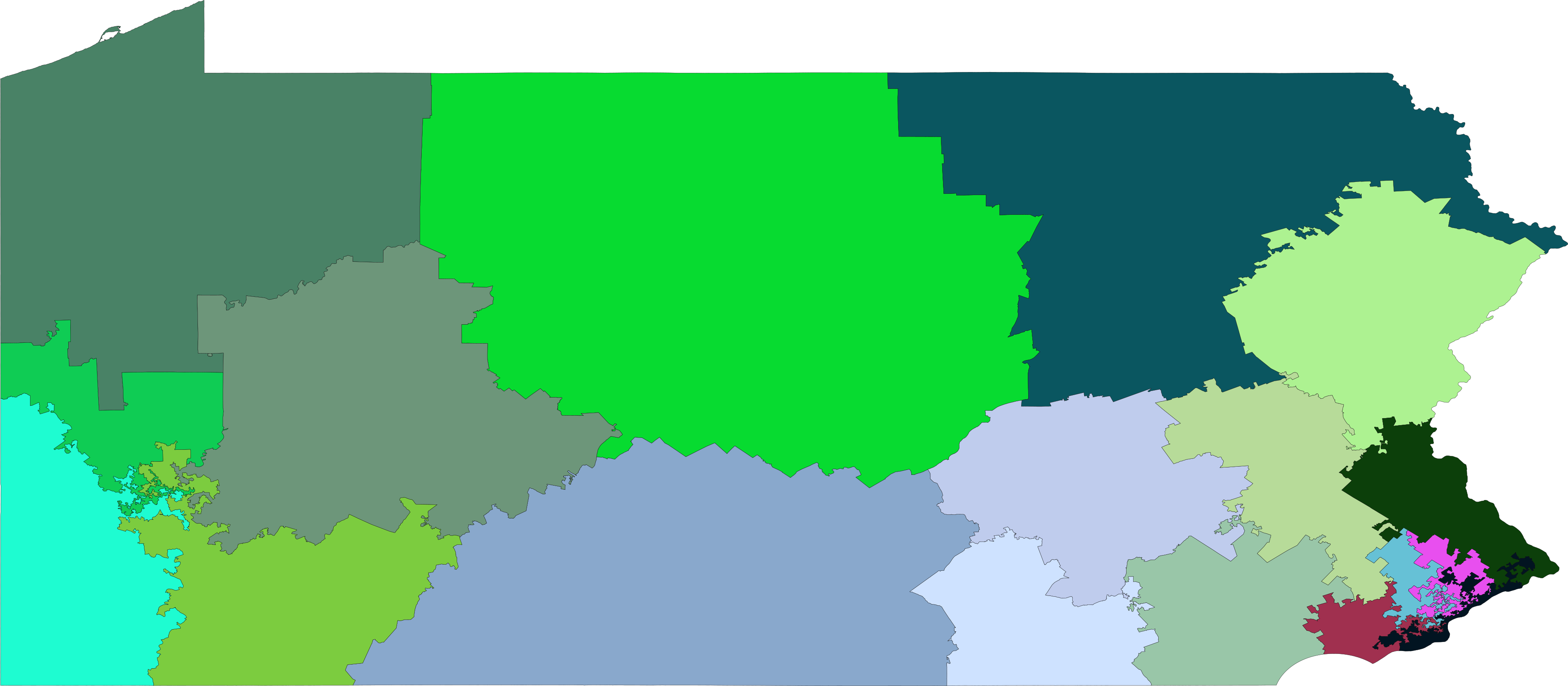}
  \includegraphics[width=.24\linewidth]{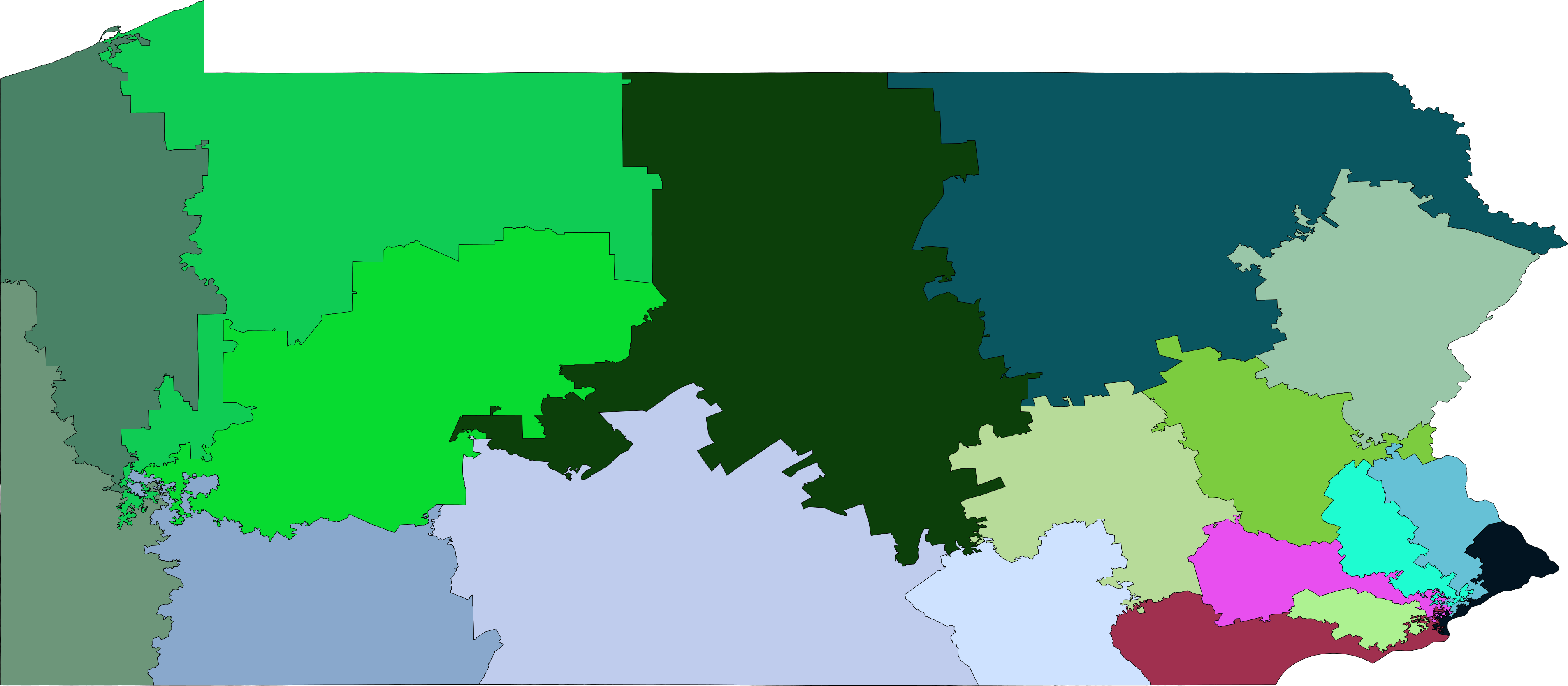}
  \includegraphics[width=.24\linewidth]{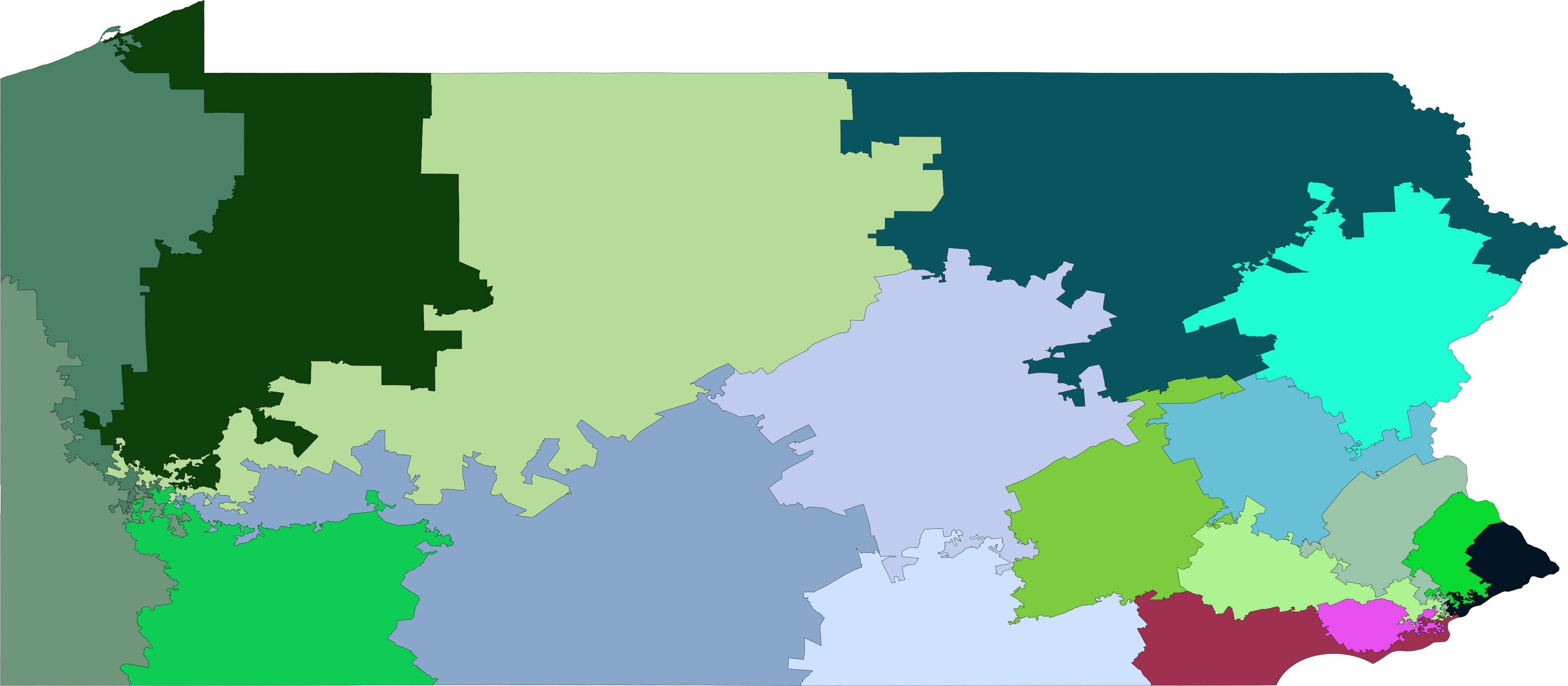}
  \includegraphics[width=.24\linewidth]{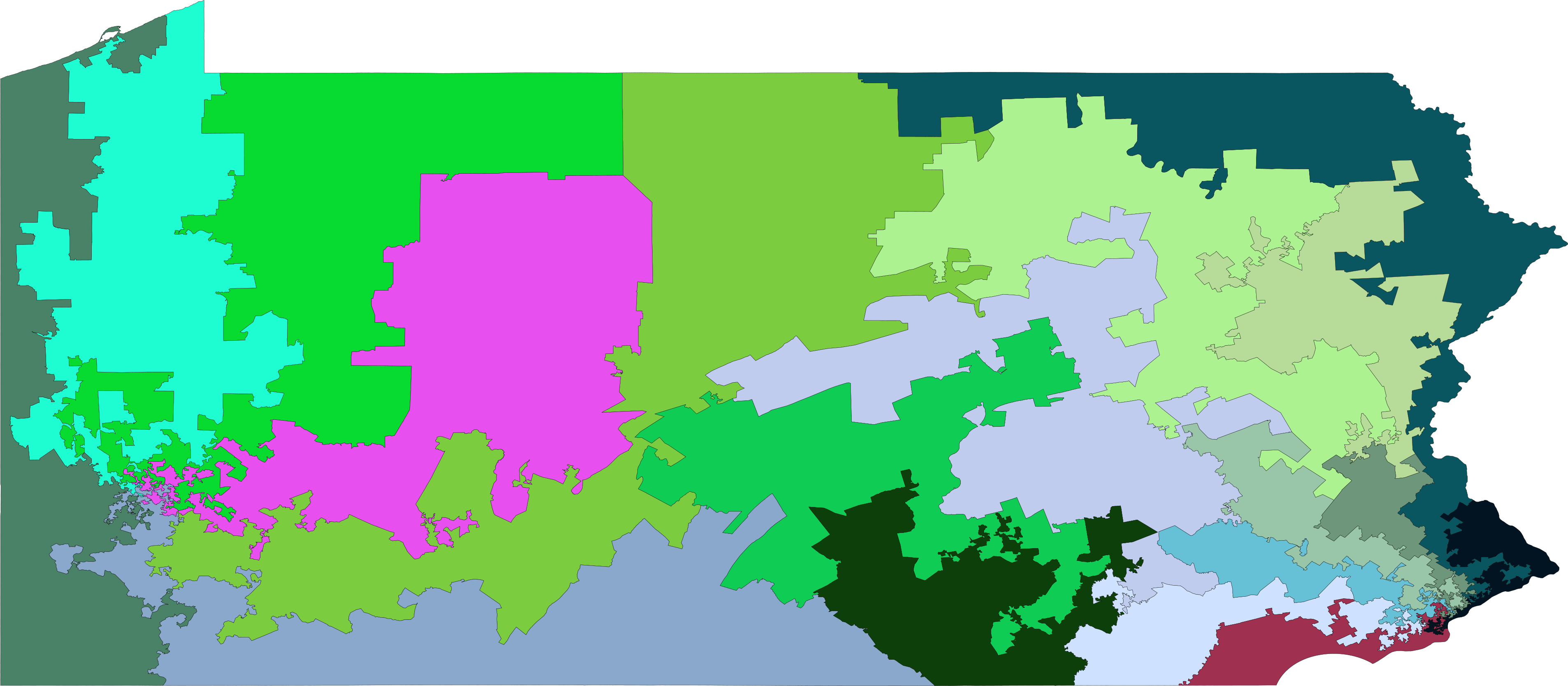}
  \end{center}
  \caption{\label{f.final} The last state from each of the above runs of the chain (perimeter, $L^1$, $L^2$, and $L^\infty$, respectively).  Note that the $L^\infty$ districting is quite ugly; with this notion of validity, every district among the 18 is allowed to be as noncompact as the worst district in the current Pennsylvania districting.  The perimeter constraint produces a districting which appears clean at a large scale, but allows rather messy city districts, since they contribute only moderately to the perimeter anyways.  The $L^1$ and $L^2$ constraints are more balanced.  Note that none of these districtings should be expected to be geometrically ``nicer'' than the current districting of Pennsylvania.  Indeed, the point of our Markov Chain framework is to compare the present districting of Pennsylvania with other ``just as bad'' districtings, to observe that even among this set, the present districting is atypical.}
\end{figure}
\section{An example where $p\approx \sqrt \e$ is best possible}

It might be natural to suspect that observing that a presented state $\sigma$ is an $\e$-outlier on a random trajectory $\sigma$ is significant something like $p\approx \e$ rather than the $p\approx \sqrt \e$ established by Theorem I.1.  However, since Theorem I.1 places no demand on the mixing rate of $\cM$, it should instead seem remarkable that any significance can be shown in general, and indeed, we show by example in this section that significance at $p\approx \sqrt \e$ is essentially best possible.

 Let $N$ be some large integer.  We let $\cM$ be the Markov chain where $X_0$ is distributed uniformly in $[0,1,2,\dots,N-1]$, and, for any $i\geq 1$, $X_i$ is equal to $X_{i-1}+\zeta_i$ computed modulo $N$, where $\zeta_i$ is $1$ or $-1$ with probability $\tfrac 1 2$.  Note that the chain is stationary and reversible.

If $N$ is chosen large relative to $k$, then with probability arbitrarily close to $1$, the value of $X_0$ is at distance greater than $k$ from 0 (in both directions).  Conditioning on this event, we have that $X_0$ is minimum among $X_0,\dots,X_k$ if and only if all the partial sums $\sum_{i=1}^j \zeta_i$ are positive.  This is just the probability that a $k$-step 1-dimensional random walk from the origin takes a first step to the right and does not return to the origin.  This is a classical problem in random walks, which can be solved using the reflection principle.

In particular, for $k$ even, the probability is given by
\[
\frac 1 {2^{k+1}} \binom{k}{k/2}\sim \frac{1}{\sqrt{2\pi k}}.
\]
Since being the minimum out of $X_0,\dots,X_k$ corresponds to being an $\e$-outlier for $\e=\tfrac{1}{k+1}$, this example shows that the probability of being an  $\e$-outlier can be as high as $\sqrt{\e/2\pi}$.

The best possible value of the constant in the $\sqrt{\e}$ test appears to be an interesting problem for future work.

\section{Notes on statistical power}
The effectiveness of the $\sqrt \e$ test depends on the availability of a good choice for $\om$, and the ability to run the test for long enough (in other words, choose $k$ large enough) to detect that the presented state is a local outlier.

It is possible, however, to make a general statement about the power of the test when $k$ is chosen large relative to the actual mixing time of the chain.  Recall that one potential application of the test is in situations where the mixing time of the chain is actually accessible through reasonable computational resources, even though this can't be proved rigorously, because theoretical bounds on the mixing time are not available.  In particular, we do know that the test is very likely to succeed when $k$ is sufficiently large, and $\om(\s_0)$ is atypical.

\begin{theorem}\label{t.power}
  Let $\cM$ be a reversible Markov Chain on $\Sigma$, and let $\om:\Sigma\to \Re$ be arbitrary.   Given $\s_0$, suppose that for a random state $\sigma\sim \pi$, $\Pr(\om(\s)\leq \om(\s_0))\leq \e$.  
  Then with probability at least
  \[
\rho\geq 1-\left(1+\frac{\e k}{10\t_2}\right)\frac{1}{\sqrt {\pmin}} \exp\left(\frac{-\e^2k}{20\t_2}\right)
  \]
$\om(\sigma)$ is an $2\e$-outlier among $\om(\s_0),\om(\s_1),\dots,\om(\s_k)$, where $\s_0,\s_1,\dots$ is a random trajectory starting from $\s_0$.
\end{theorem}
Here $\t_2$ is the \emph{relaxation time} for $\cM$, defined as $1/(1-\l_2)$, where $\l_2$ is the second eigenvalue of $\cM$.   $\t_2$ is thus the inverse of the spectral gap for $\cM$, and is intimately related to the mixing time of $\cM$\cite{Pbook,Abook,Fbook}.
The probability $\r$ in Theorem \ref{t.power} converges exponentially quickly to $1$, and, moreover, is very close to 1 once $k$ is large relative to $\t_2$.  In particular, Theorem \ref{t.power} shows that the $\sqrt \e$ test \emph{will work} when the test is run for long enough.  Of course, one strength of the $\sqrt \e$ test is that it can sometimes demonstrate bias even when $k$ far too small to mix the chain, as is almost certainly the case for our application to gerrymandering.  When these short-$k$ runs are successful at detecting bias is of course dependent on the relationship of the presented state $\sigma_0$ and its local neighborhood in the chain.

Theorem \ref{t.power} is an application of the following theorem of Gillman:
\begin{theorem}\label{t.gillman}
  Let $\cM=X_0,X_1,\dots$ be a reversible Markov Chain on $\Sigma$, let $A\sbs \Sigma$, and let $N_k(A)$ denote the number of visits to $A$ among $X_0,\dots,X_k$.  Then for any $\g>0$,
  \[
\Pr(N_k(A)/n-\pi(A)>\g)\leq \left(1+\frac{\g n}{10\t_2}\right)\sqrt{\sum_{\s}\frac{\Pr(X_0=\sigma)^2}{\pi(\sigma)}} \exp\left(\frac{-\g^2n}{20\t_2}\right).
\]
\end{theorem}

\begin{proof}[Proof of Theorem \ref{t.power}]
  We apply Theorem \ref{t.gillman}, with $A$ as the set of states $\s\in \Sigma$ such that $\om(\s)\leq \om(\s_0)$, with $X_0=\s_0$, and with $\g=\e$.  By assumption, $\pi(A)\leq \e$, and Theorem \ref{t.gillman} gives that
  \[
  \Pr(N_k(A)/k>2\e)\leq \left(1+\frac{\e k}{10\t_2}\right)\sqrt{\frac{1}{\pmin}} \exp\left(\frac{-\e^2k}{20\t_2}\right).
\]
\end{proof}

\section{A result for small variation distance}
In this section, we give a corollary of Theorem I.1 which applies to the setting where $X_0$ is distributed not as a stationary distribution $\pi$, but instead with small total variation distance to $\pi$.

The {\em total variation distance} $\tv{\r_1-\r_2}$ between probability distributions $\r_1,\r_2$ on a probability space $\Omega$ is defined to be 
\begin{equation}
  \label{corTV}
\tv{\r_1-\r_2}:=\sup_{E\subseteq \Omega}|\r_1(E)-\r_2(E)|.
\end{equation}
\begin{corollary}\label{cor1}
Let $\cM=X_0,X_1,\dots$ be a reversible Markov chain with a stationary distribution $\p$, and suppose the states of $\cM$ have real-valued labels.  If $\tv{X_0-\p}\leq \e_1$, then for any fixed $k$, the probability that the label of $X_0$ is an $\e$-outlier from among the list of labels observed in the trajectory $X_0,X_1,X_2,\dots,X_k$ is at most $\sqrt{2\e}+\e_1$.
\end{corollary}
The theorem is intuitively clear; we provide a formal proof below for completeness.

\begin{proof}
  If $\r_1,\r_2,$ and $\t$ are probability distributions, then we have that the product distributions $(\r_1,\t)$ and $(\r_2,\t)$ satisfy
  \begin{equation}
    \label{tve}
  \tv{(\r_1,\t)-(\r_2,\t)}=\tv{\r_1-\r_2}.
  \end{equation}
Our plan now is to split the randomness in the trajectory $X_0,\dots,X_k$ of the Markov Chain into two independent sources: the initial distribution is $X_0\sim \r$, and $\tau$ is the uniform distribution on sequences of length $k$ of real numbers $r_1,r_2,\dots,r_k$ in $[0,1]$.  We can view the distribution of the trajectory $X_0,X_1,\dots,X_k$ as the product $(\r,\t)$ by using sequences of reals $r_1,\dots,r_k$ to choose transitions in the chain; from $X_i=\s_i$, if there are $L$ transitions possible, with probabilities $p_1,\dots,p_L$, then we make the $t$th possible transition if $r_i\in [p_1+\dots+p_{t-1},p_1+\dots+p_{t-1}+p_t)$.

Now we have from \eqref{tve} that if $\tv{\r-\pi}\leq \e_1$, then $\tv{(\r,\t)-(\pi,\t)}\leq \e_1$.  Therefore, any event which would happen with probability at most $p$ for the sequence $X_0,\dots,X_k$ when $X_0\sim \pi$ must happen with probability at most $p+\e_1$ when $X_0\sim \r$ where $\tv{\r-\pi}\leq \e_1$.  This gives the Corollary.
\end{proof}

\end{document}